\documentclass[12pt]{amsart}

\usepackage{amssymb}
\usepackage{amsthm}
\usepackage[mathscr]{eucal}
\usepackage{enumitem}

\newtheorem{Theorem}{Theorem}[section]
\newtheorem{theorem}[Theorem]{Theorem}
\newtheorem{Proposition}[Theorem]{Proposition}
\newtheorem{proposition}[Theorem]{Proposition}

\newtheorem{corollary}[Theorem]{Corollary}
\newtheorem{Lemma}[Theorem]{Lemma}
\newtheorem{lemma}[Theorem]{Lemma}

\newtheorem{fact}[Theorem]{Fact}

\newtheorem{claim}[Theorem]{Claim}

\theoremstyle{definition}

\newtheorem{remark/def}[Theorem]{Remark/Definition}
\newtheorem{Definition}[Theorem]{Definition}
\newtheorem{definition}[Theorem]{Definition}

\newtheorem{notation}[Theorem]{Notation}

\newsavebox{\indbin}
\savebox{\indbin}{\begin{picture}(0,0)
\newlength{\gnu}
\settowidth{\gnu}{$\smile$} \setlength{\unitlength}{.5\gnu}
\put(-1,-.65){$\smile$} \put(-.25,.1){$|$}
\end{picture}}

\newcommand{\be}{\begin{enumerate}}
\newcommand{\bd}{\partial}

\newcommand{\bt}{\begin{theorem}}
\newcommand{\bl}{\begin{lemma}}
\newcommand{\ee}{\end{enumerate}}
\newcommand{\ed}{\end{defn}}
\newcommand{\et}{\end{theorem}}
\newcommand{\el}{\end{lemma}}

\newcommand{\la}{\langle}
\newcommand{\ra}{\rangle}

\newcommand{\ov}{\overline}

\newcommand{\F}{\varphi}

\newcommand{\CP}{{\mathcal P}}

\newcommand{\CB}{{\mathcal B}}
\newcommand{\CC}{{\mathcal C}}

\newcommand{\CL}{{\mathcal L}}
\newcommand{\CH}{{\mathcal H}}
\newcommand{\CM}{{\mathcal M}}

\newcommand{\CZ}{{\mathcal Z}}

\newcommand{\CS}{{\mathcal S}}

\newcommand{\dom}{\mbox{dom}}

\newcommand{\Pm}{{\mathcal P}^{-}}

\newcommand{\Aut}{\operatorname{Aut}}
\newcommand{\aut}{\operatorname{Aut}}

\newcommand{\supp}{\operatorname{supp}}

\def\eq{\operatorname{eq}}

\def\dcl{\operatorname{dcl}}

\def\dom{\operatorname{dom}}

\def\acl{\operatorname{acl}}

\def\tp{\operatorname{tp}}

\def\supp{\operatorname{supp}}
\def\Bd{\partial}

\title[The Hurewicz correspondence]{Homology groups of types in
stable theories and the Hurewicz correspondence}

\author{John Goodrick}
\author{Byunghan Kim}
\author{Alexei Kolesnikov}

\address{Department of Mathematics\\ Universidad de los Andes \\
Bogot\'{a}, Colombia}
\address{Department of Mathematics\\ Yonsei University\\
Seoul,  Korea}
\address{Department of Mathematics\\ Towson University, MD\\
USA}

\email{jr.goodrick427@uniandes.edu.co}
\email{bkim@yonsei.ac.kr}
\email{akolesnikov@towson.edu}

\thanks{The second  author was supported by NRF of Korea grant  2013R1A1A2073702, and Samsung
Science Technology Foundation under Project Number SSTF-BA1301-03.}

\begin{document}

\begin{abstract}
We give an explicit description of the homology group $H_n(p)$ of a strong type $p$ in any stable theory under the assumption that for every non-forking extension $q$ of $p$ the groups $H_i(q)$ are trivial for $2\le i<n$. The group $H_n(p)$ turns out to be isomorphic to the automorphism group of a certain part of the algebraic closure of $n$ independent realizations of $p$; it follows from the authors' earlier work that such a group must be abelian. We call this the ``Hurewicz correspondence'' by analogy with the Hurewicz Theorem in algebraic topology. 
\end{abstract}

\maketitle

The present paper is a part of the project to study type amalgamation properties in first-order theories by means of homology groups of types. Roughly speaking (more precise definitions are recalled below in Section~$1$), a strong type $p$ is said to have \emph{$n$-amalgamation} if commuting systems of elementary embeddings among algebraic closures of proper subsets of the set of $n$ independent realizations of $p$ can always be extended to the algebraic closure of all $n$ realizations. The type $p$ has \emph{$n$-uniqueness} if this extension is essentially unique. Generalized amalgamation properties for systems of \emph{models} were introduced by Shelah in~\cite{Sh87ab} and played an important role in~\cite{Sh:c}. The type amalgamation properties were studied extensively by Hrushovski in \cite{Hr} and applications were given. In fact, the type amalgamation properties have been used in model theory at least as far back as Hrushovski's classification of trivial totally categorical theories in \cite{tot_cat_struct}.

In the previous paper \cite{GKK2}, we introduced a notion of homology groups for a complete strong type in any stable, or even rosy, first-order theory. The idea was that these homology groups should measure information about the amalgamation properties of the type $p$. We proved that if $p$ has $n$-amalgamation for all $n$, then $H_n(p) = 0$ for every $n$, and that the failure of $4$-amalgamation (equivalently, the failure of $3$-uniqueness) over $\dom(p)$ in a stable theory corresponds to non-triviality of the group $H_2(p)$. Furthermore, we established in that paper that $H_2(p)$ is isomorphic in stable theories to a certain automorphism group of closures of realizations of $p$. In \cite{GKK4}, the theory is developed in a more general context of {\em amenable collection of functors}.

The current paper generalizes the main results of \cite{GKK2} in the stable context: if the type $p$ does not have $(n+2)$-amalgamation, then for some $i$ with $2 \leq i \leq n$ and some nonforking extension $p'$ of $p$, the group $H_i(p')$ must be nonzero. Furthermore, at the first $i$ for which such $p'$ with $H_i(p') \neq 0$ exists, we show that $H_i(p')$ is isomorphic to a certain automorphism group $\Gamma_i(p')$, which immediately implies that $H_i(p')$ is a profinite group.

To structure the proof in a more transparent way, we state a technical lemma (Lemma~\ref{technical lemma}) in Section 2 of this paper and prove the main result (Theorem~\ref{main}) of the paper using the lemma. The proof of Lemma~\ref{technical lemma} uses certain algebraic structures ($n$-polygroupoids) that were linked to failure of $(n+2)$-amalgamation in the previous paper~\cite{GKK3}; but neither the statement of the lemma, nor the proof of Theorem~\ref{main} from the lemma do not use these structures. 

The proof of Lemma~\ref{technical lemma} is contained in Section 3 of the paper. It turns out that most of the results of~\cite{GKK3} do not need global $n$-amalgamation assumptions, only the amalgamation properties for the type $p$ and its non-forking extensions. The additional work to verify the results of ~\cite{GKK3} is contained in Section 4.

\section{Notation and preliminaries}

{\bf In this paper, we always work with a fixed complete stable theory $T=T^{eq}$ in a language $\CL$,  its saturated model $\CM=\CM^{\eq}$, and a complete type $p$ 
over a small set $B=\acl(B)$.} Throughout this paper, independence is nonforking independence. We use the usual notational conventions of stability theory, plus some conventions from \cite{GKK2} and \cite{GKK3}. We summarize the list of notations at the end of this section.

\subsection{Definition of the homology groups}

We start by recalling from \cite{GKK} and \cite{GKK2} the definition of homology groups of the type $p$, and the basic notions of the amalgamation properties related to the computation of  the homology groups.

A family of sets $X$ ordered by inclusion can be endowed with a natural poset category structure: the objects are the elements of $X$ and the morphisms are single inclusion maps $\iota_{u,v} : u \rightarrow v$ between any two sets $u,v\in X$ with $u \subseteq v$.  The set $X$ is called \emph{downward-closed} if whenever $u \subseteq v \in X$, then $u \in X$. 

We let $\CC_B$  denote the category of all small algebraically closed subsets  of $\CM$ containing $B$, where morphisms are  elementary maps over $B$ (i.e., fixing $B$ pointwise).   
For a downward closed $X$ and a functor $f:X\to \CC_B$ and $u\subseteq v\in X$, we write $f^u_v:=f(\iota_{u,v})$ and $f^u_v(u):=f^u_v(f(u))\subseteq f(v)$. 

Simplicial structure on the sets of algebraic closures of independent sequences realizing a type is contained in Definitions~\ref{n-simplex} and~\ref{n-boundary} below. We begin by recalling an auxiliary definition.

\begin{definition}
A \emph{(closed independent) $p$-functor} is a functor $f: X \to \CC_B$ such that:

\begin{enumerate}

\item For some finite $s \subseteq \omega$, $X$ is a downward-closed subset of $\CP(s)$; 

\item $f(\emptyset) \supseteq  B$; and for $i\in s$, $f(\{i\})$ (if it is defined) is of the form $\acl(Cb)$ where
 $b(\models p)$ is independent with $C=f^\emptyset_{\{i\}}(\emptyset)$ over $B$.

\item
For all non-empty $u\in X$, we have $f(u) = \acl(B \cup \bigcup_{i\in u} f^{\{i\}}_u(\{i\}))$ and the set $\{f^{\{i\}}_u(\{i\}) : i \in u \}$ is independent over $f^\emptyset_u(\emptyset)$.
\end{enumerate}

If $f(\emptyset)=B$ (so for any $u\in X$,  $f^\emptyset_{u}(\emptyset)=B$) then we say  $f$ is {\em over} $B$.
\end{definition}

\begin{definition}
\label{n-simplex}
Let $n\ge 0$ be a natural number. An \emph{$n$-simplex in $p$} is a  $p$-functor $f : \CP(s) \rightarrow \CC_B$ for some set $s \subseteq \omega$ with $|s| = n+1$.  The set $s$ is called the \emph{support of $f$}, or $\supp(f)$.

Let $\CS_n(p)$ denote the collection of all $n$-simplices {\em over $B$} in $p$; and
let $\CC_n (p)$ denote the free abelian group generated by $\CS_n(p)$; its elements are called \emph{$n$-chains in $p$}.  
Similarly, we define $\CS(p):=\bigcup_{n} \CS_n(p)$, and $\CC(p):=\bigcup_{n} \CC_n(p)$. 
The \emph{support of a chain $c$} is the union of the supports of all the simplices that appear in $c$ with a nonzero coefficient.
\end{definition}

The boundary maps are defined to be the usual simplicial boundary embeddings.

\begin{definition}
\label{n-boundary}
If $n \geq 1$ and $0 \leq i \leq n$, then the \emph{$i$-th boundary operator} $\bd^i_n : \CC_n (p) \rightarrow \CC_{n-1} (p)$ is defined so that if $f$ is an $n$-simplex in $p$ with domain $\CP(s)$, where $s = \{s_0, \ldots, s_n\}$ with $s_0 < \ldots < s_n$, then $$\bd^i_n(f) = f \upharpoonright \CP(s \setminus \{s_i\}),$$ and we extend $\bd^i_n$ linearly to a group map on all of $\CC_n (p)$.

If $n \geq 1$ and $0 \leq i \leq n$, then the \emph{boundary map} $\bd_n : \CC_n(p) \rightarrow \CC_{n-1}(p)$ is defined by the rule
$$\bd_n(c) = \sum_{0 \leq i \leq n} (-1)^i \bd^i_n (c).$$
We write $\bd^i$ and $\bd$ for $\bd^i_n$ and $\bd_n$, respectively, if $n$ is clear from context.

The kernel of $\bd_n$ is denoted $\CZ_n(p)$, and its elements are called \emph{$n$-cycles}.  The image of $\bd_{n+1}$ in $\CC_n(p)$ is denoted $\CB_n(p)$. The elements of $\CB_n(p)$ are called \emph{$n$-boundaries}.
\end{definition}

It can be shown (by the usual combinatorial argument) that $\CB_n(p) \subseteq \CZ_n
(p)$, or more briefly, ``$\bd_n\circ \bd _{n+1} = 0$.''  Therefore we can define simplicial homology groups in the type $p$:

\begin{definition}
The \emph{$n$th (simplicial) homology group of the type $p \in S(B)$} is $$H_n(p) := \CZ_n(p) / \CB_n(p).$$
\end{definition}

Finally, we define the amalgamation properties of the type $p$.  As usual,
$n=\{0,\dots,n-1\}$.

\begin{definition}
Let $n\ge 1$.

(1)\
We say $p$ has \emph{$n$-amalgamation} (or \emph{$n$-existence}) if for any  $p$-functor $f : \CP^-(n)(:=\CP(n)\setminus \{n\}) \rightarrow \CC_B$, there is an $(n-1)$-simplex $g$ in $p$ such that $g \supseteq f$.
 When $f$ above ranges only over $p$-functors over $B$, we say
  $p$ has \emph{$n$-amalgamation over $B$.}

If $p$ has $k$-amalgamation for every $k$ with $1 \leq k \leq n$, then we say $p$ has
\emph{$n$-complete amalgamation} (or \emph{$n$-CA} for short).

(2)\
We say that $p \in S(B)$ has \emph{$n$-uniqueness} if for any closed independent $p$-functor $f: \CP^-(n) \rightarrow \CC_B$ and any two $(n-1)$-simplices $g_1$ and $g_2$ in $p$ extending $f$, there is a natural isomorphism $F: g_1 \rightarrow g_2$ such that $F \upharpoonright \dom(f)$ is the identity. 
Similarly we say
  $p$ has \emph{$n$-uniqueness over $B$} when $f$ above ranges over $p$-functors over   $B$.

If $p$ has $k$-uniqueness for every $k$ with $2 \leq k \leq n$, then we say $p$ has
{\em $(\leq n)$-uniqueness}.

\end{definition}

It follows directly from  the definitions above that the properties of $n$-uniqueness and $n$-existence of $p$ are preserved under nonforking extensions. Trivially,  $1$-amalgamation holds in $p$. The extension axiom and stationarity of nonforking imply
$2$-amalgamation, and $1$- and  $2$-uniqueness of $p$, respectively. In general, the following holds.

\begin{fact} \cite{GKK}\label{trivialgamma}
 Let $n\geq 1$. Then 
$p$ has $(\leq n)$-uniqueness if and only if it has $(n+1)$-CA. 

More precisely, assume $p$ has  $(\leq n)$-uniqueness. Then the following are equivalent:
\begin{enumerate}
\item
$p$ has $(n+1)$-uniqueness over $B$;
\item
$p$ has $(n+2)$-amalgamation over $B$;
\item
$\Gamma_n(p)=0$.
\end{enumerate}
\end{fact}

The following fact is proved in \cite{K}:

\begin{fact}
\label{uniq_powers}
For any $k, n \geq 1$, $p$ has $n$-uniqueness if and only if $p^{(k)}$ has $n$-uniqueness (where $p^{(k)}$ is the complete type of $k$ independent realizations of $p$ over the base set $B$).
\end{fact}

We now introduce a particular class of $n$-cycles representing all the members of  $H_n(p)$ under the 
assumption of $(n+1)$-CA.

\begin{definition}
\label{pocket}
If $n \geq 1$, an \emph{$n$-pocket} is an $n$-chain $c$ of the form $\pm (f-g)$, where $f$ and $g$ are $n$-simplices that have the same boundary: $\bd f = \bd g$.
\end{definition}

Notice that any $n$-pocket is an $n$-cycle.   

\begin{theorem} \cite{GKK2}
\label{Hn_pockets}
If $p$ has $(\leq n)$-uniqueness for some $n \geq 1$, then $$H_n(p) = \left\{ \left[c\right] : c \textup{ is an } n\textup{-pocket in } p \textup{ with support } n+1 \right\}.$$

If the functors $f,g\in \CS_n(p)$ are naturally isomorphic and $f-g$ is an $n$-pocket, then $f-g$ is a boundary.

In any simple theory, $H_1(p)=0$.
\end{theorem}

The above theorem and Fact \ref{trivialgamma} imply Theorem \ref{main} holds for $n=1$. 

\subsection{Notation used in this paper}

Tuples of elements of the monster model or of variables will be denoted by lower-case letters (without a bar); the upper-case letters will denote sets. Throughout the paper, we fix a complete type over a small algebraically closed set $B$; this set is included in all the algebraic closures. The algebraic closures will be denoted by a bar; for example, given  a tuple $c$, the symbol $\bar c$ denotes $\acl(cB)$.

For  sets $A$ and $C$, the symbol $\aut(A/C)$ denotes the group of elementary maps  over $C$ (i.e., fixing $C$ pointwise) from $A\cup C$ onto 
$A\cup C$.  For a type $q=\tp(a/C)$ with  the solution set $A$, $\aut(q):=\aut(A/C)$.

For the fixed type $p$, the symbol $p^{(k)}$ will denote the complete type of $k$ independent realizations of the type $p$.

If $n$ is a natural number, the symbol $[n]$ will denote the set $\{1,\dots,n\}$. For a set $X$, the symbols $X^{(n)}$ and $X^{(<n)}$ denote, respectively, the set of $n$-element subsets of $X$ and the set of subsets of $X$ of size less than $n$.

The boundary symbol $\Bd$ will be used in three different contexts:
\begin{enumerate}
\item
to denote the boundary operation on $n$-simplices and $n$-chains, as described in Definition~\ref{n-boundary}; if $f$ is an $n$-simplex, then $\Bd( f)$ is an $(n-1)$-chain;
\item
to denote ``the boundary'' part of the algebraic closure of a set of $n$ independent realizations of the type $p$: if $c_1,\dots,c_{n}$ realize $p^{(n)}$, then 
$$
\Bd(c_1...c_{n}):=\dcl(\bigcup_{i=1}^{n}\ov{c_1...\hat c_i...c_{n}});
$$
\item
to denote a part of the algebraic closure of the set $f(s)$ for a simplex $f$ with the support $s$.  Namely, if $f\in \CS_{n-1}(p)$ for some $n\ge 1$ (and so $\supp(f)= \{s_1,\dots,s_n\}$ for some $n$-element set $s \subset \omega$), then $\bd[f] := \bd(c_1\dots c_n)$ where $c_i := f^{\{s_i\}}_s (\{s_i\})$. We will use the square brackets $\Bd[f]$ to separate this context from (1).
\end{enumerate}

A different part of the  algebraic closure of a set of $n$ independent realizations of the type $p$ will be denoted by a tilde: if $c_1,...,c_{n+1}$ realize $p^{(n+1)}$,
we let 
$$
\widetilde{c_1...c_{n}}:=\ov{c_1...c_{n}}\cap \dcl(\bigcup^{n}_{i=1}\ov{c_1...\hat c_i...c_{n+1}}).
$$
Stationarity guarantees that the set $\widetilde{c_1...c_{n}}$ does not depend on the choice of the element $c_{n+1}$.
If $f$ is an $(n-1)$-simplex in $p$ with the support $\{s_1,\dots,s_n\}$ and $c_i = f^{\{s_i\}}_s (\{s_i\})$, then $\widetilde{f}$ will denote the set $\widetilde{c_1...c_{n}}$.

Finally, define
$$\Gamma_n(p):=\aut(\widetilde{c_{1}...c_n}/ \Bd(c_1...c_n)),$$
where $c_1,...,c_n \models p^{(n)}$. Since $p$ is stationary, it is routine to check that this definition does not depend on the choice of the realizations $c_i$. Equivalently, $\Gamma_n(p) = \aut(\widetilde{f}/\Bd[f])$ for some (any) simplex $f\in \CS_{n-1}(p)$.

\section{Main result}

The main result of the paper is the following.

\begin{theorem}\label{main} Let $n\geq 1$. 
Let $T=T^{eq}$ be a stable theory and let $p$ be a strong type in $T$. Assume that $p$ has $k$-amalgamation for each $k \in \{2, \ldots, n+1\}$. Then $H_n(p) \cong \Gamma_n(p)$; the latter group is always a profinite abelian group.
\end{theorem}

In \cite{GKK2}, the theorem for $n=1,2$ was shown, and the above generalization was conjectured. We call this the Hurewicz correspondence  since the result connects the homology groups of $p$ to $\Gamma_n(p)$, which are  similar to  homotopy groups, as in algebraic topology (see \cite{B}). In particular, it is shown in \cite{GKK} that $\Gamma_2(p)$ is the profinite limit of the vertex groups of relatively definable groupoids obtained from the failure of $4$-amalgamation of the type. Since the definable vertex groups consist of something like ``homotopy equivalence of paths,'' as described in  \cite{GK}, the limit $\Gamma_2(p)$ of definable vertex groups is analogous to  the fundamental groupoid  $\pi_1$ of 
the type, as in singular homology theory (see, for example, \cite{Br}). Note that there is a mismatch in the numbering: our group $\Gamma_n$ corresponds to the group $\pi_{n-1}$ in algebraic topology.

To define a map from $H_n(p)$ to $\Gamma_n(p)$ and show that it is an isomorphism, we will use Lemma~\ref{technical lemma}. The lemma shows that we can select, for each $(n-1)$-simplex $f$, a complete $*$-type $\Pi f$ over the set $\bd[f]$ and a realization $\alpha(f)$ of $\Pi f$;  the tuples $\alpha(f)$ control $\widetilde{f}$ in the sense that $\dcl(\alpha(f))=\widetilde{f}$; that the group $\Aut(\widetilde{f}/\Bd[f])$ acts regularly and transitively on the set of realizations of $\Pi f$. These automorphism groups for different simplices turn out to be canonically isomorphic to a certain group $G$ constructed in the proof of the lemma; this group turns out to be abelian. 

The selector function $\alpha$ and the property in item (4) will allow us to assign an element of the group $G$ to every $n$-simplex (and extend the assignment to chains by linearity). This assignment does depend on a particular choice of the function $\alpha$. However, the key properties shown in items (5) and (6) allow us to establish that for any choice of the function $\alpha$, the group element assigned to the boundary of an $(n+1)$-simplex has to be 0.
These properties also allow us to characterize when the union of automorphisms of the $(n-1)$-faces of an $n$-simplex can be lifted to an automorphism of the entire $n$-simplex. The statement of the lemma does not use the language of polygroupoids, but these objects play a central role in its proof. We defer the proof of the lemma to Section~\ref{s:technical lemma}; but show how Theorem~\ref{main} follows from it.

Before stating the lemma, let us make a number of notation agreements. 

\begin{notation}
If $g\in \CS_n(p)$ has the support $t=\{t_0,\dots,t_n\}$ listed in increasing order and $i\in n+1$, then $\Pi_i g $ will denote the image of the type $\Pi(\bd^i g)$ under the transition map $g^{t\setminus \{t_i\}}_t$. Thus if  $c_i := g^{\{t_i\}}_t(\{t_i\})$ (so that $g(t) = \ov{ c_0 \ldots c_n}$), then $\Pi_i g$ is a complete $*$-type over $\bd(c_0, \ldots, \widehat{c_i}, \ldots, c_n)$.

The symbol $\alpha_i(g)$ will denote the image of the tuple $\alpha(\bd ^i g)$ under the same transition map  $g^{t\setminus \{t_i\}}_t$. Hence $\alpha_i(g)$ realizes the type $\Pi_i g$.

If $h\in \CS_{n+1}(p)$, with the support $s=\{s_0,\dots,s_{n+1}\}$ (listed in increasing order) and $0\le i<j\le n+1$, then $\Pi_{i,j} h$ will denote the image of the type $\Pi(\bd^i(\bd^j h))$ under the transition map $h^{s\setminus \{s_i,s_j\}}_s$. Similarly, $\alpha_{i,j}(h)$ will denote the image of the  tuple $\alpha(\bd ^i(\bd ^j h))$ under $h^{s\setminus \{s_i,s_j\}}_s$.

For the remainder of this section, we will use the letters $f$, $g$, and $h$ to denote simplices in $\CS_{n-1}(p)$, $\CS_{n}(p)$, and $\CS_{n+1}(p)$, respectively. The Greek letters will be used to denote the elements of the group $G$.
\end{notation}

\begin{lemma}
\label{technical lemma}
Suppose $T=T^{\eq}$ is a stable theory, $B$ is an algebraically closed subset in the monster model $\CM$ of $T$, and a complete type $p\in S(B)$ has $\le n$-uniqueness but fails $(n+1)$-uniqueness.
There is a selector function $(\Pi,\alpha)$ on $\CS_{n-1}(p)$ that produces, for every $(n-1)$-simplex $f\in \CS_{n-1}(p)$ a $*$-type $\Pi f$ over $\bd[f]$ and a tuple $\alpha(f) \models \Pi f$, as well as 
a complete $*$-type $q(x_0,\dots, x_{n})$ over $B$ and a profinite abelian group $G$ such that:
\begin{enumerate}
\item
for every $f\in \CS_{n-1}(p)$, $\alpha(f)$ is a (possibly infinite) tuple of elements in $\widetilde{f}$;
\item
$\widetilde{f} = \dcl(w)$ for every $w\models \Pi f $;
\item
\label{action}
for each $f\in \CS_{n-1}(p)$, there is a regular transitive action of $G$ on the set of realizations of $\Pi f $. The action is $*$-type definable in the following sense: for every $\gamma\in G$, there is a unique $*$-type $r_{\gamma}(x,y)$ over $B$ such that $\models r_{\gamma}(w,w')$ if and only if $w' = \gamma. w$;
\item
\label{connected polygroupoid}
if $g\in \CS_{n}(p)$, then for any realizations $w_i$ of $\Pi_i g$ for $i=0,\dots,n$, there is a unique element $\gamma\in G$ such that 
$\models q(w_0,\dots,w_{n-1},\gamma.w_{n})$;
\item
if $\models q(w_0,\dots,w_{n})$ and $\gamma_0,\dots,\gamma_{n}$ are elements of $G$, then $\models q(\gamma_0.w_0,\dots,\gamma_{n}.w_{n})$ if and only if $\sum_{i=0}^{n} (-1)^{i}\gamma_i = 0$;
\item
\label{associativity polygroupoid}
if $h\in \CS_{n+1}(p)$ and $\{w_{i,j} \models \Pi_{i,j}h \mid 0\le i<j\le n+1\}$ are such that
$$
\models q(w_{0,k},\dots,w_{k-1,k},w_{k,k+1},\dots,w_{k,n+1}) \textrm{ for }0\le k\le n,
$$
then $\models q(w_{0,n+1},\dots,w_{n,n+1})$.
\end{enumerate}
\end{lemma}

\begin{corollary}
The group $G$ is isomorphic to the group $\Gamma_n$.
\end{corollary}

\begin{proof}
By (2) and (3) in Lemma~\ref{technical lemma}, the action by $g\in G$ on $\Pi(f)$ induces an automorphism in $\Gamma_n$. Since the action is regular and transitive, this correspondence is an isomorphism.
\end{proof}

Thus, we need to establish an isomorphism between $H_n(p)$ and $G$, and item (\ref{connected polygroupoid}) in Lemma~\ref{technical lemma} offers a way to do this. 

\begin{notation}
For $g\in \CS_n(p)$, let  $\varepsilon(g)\in G$ be the unique element such that $\models q(\alpha_1(g),\dots,\alpha_{n}(g), \varepsilon(g). \alpha_{n+1}(g))$. Thus we have a well-defined function $\varepsilon:\CS_n(p)\to G$. Extending $\varepsilon$ to $\CC_n(p)$ by linearity, we get a function $\varepsilon:\CC_n(p)\to G$.
\end{notation}

 It is clear that $\varepsilon$ is a homomorphism from $\CC_n(p)$ to $G$. In the three lemmas below, we establish that
\begin{itemize}
\item
If $d\in \CB_n(p)$, then $\varepsilon(d)=0$. This will show that the function $[c]\in H_n(p)\mapsto \varepsilon(c)\in G$ is well-defined; we will use the symbol $\ov{\varepsilon}$ for this function.
\item
The homomorphism $\ov{\varepsilon}$ is injective, that is, if $d\in \CZ_n(p)$ and $\varepsilon(d)=0$, then $d\in \CB_n(p)$.
\item
The homomorphism $\ov{\varepsilon}$ is surjective: for every $\gamma\in G$, there is $d\in \CZ_n(p)$ such that $\varepsilon(d)=\gamma$.
\end{itemize}

\begin{lemma}
\label{lemma a}
If $d\in \CB_n(p)$, then $\varepsilon(d)=0$.
\end{lemma}

\begin{proof}
Since $\varepsilon$ is a linear function, it suffices to establish the claim of the lemma for the case when $d=\Bd h$, where $h\in \CS_{n+1}(p)$.
In this case, $d=\sum_{j\leq n+1}(-1)^j g_j$ where $g_j:=\Bd^j h$; and we need to show that $\sum_{j\leq n+1}(-1)^j \varepsilon(g_j)=0$.

Taking the elements $\alpha_{i,j}(h)$ (these are the elements described just before the statement of Lemma~\ref{technical lemma}), we have:
$$
\models q(\alpha_{0,k}(h),\dots,\alpha_{k-1,k}(h),\alpha_{k,k+1}(h),\dots,\varepsilon(g_k).\alpha_{k,n+1}(h)) \textrm{ for }0\le k\le n
$$
by definition of the function $\varepsilon$. Lemma~\ref{technical lemma}(\ref{associativity polygroupoid}) now gives
$$
\models q(\varepsilon(g_0).\alpha_{0,n+1}(h),\dots,\varepsilon(g_{n}).\alpha_{n,n+1}(h))
$$
On the other hand $\models q(\alpha_{0,n+1}(h),\dots,\varepsilon(g_{n+1}).\alpha_{n,n+1}(h))$ and thus 
$$
\sum_{i=0}^n (-1)^{i}\varepsilon(g_i) = (-1)^n \varepsilon(g_{n+1}).
$$
The latter equality gives the needed $\sum_{i\leq n+1}(-1)^i \varepsilon(g_i)=0$.
\end{proof}

Before we prove the injectivity of $\ov{\varepsilon}$, we need to establish the following proposition.

\begin{proposition}
\label{epsilon_prop}
Suppose that $g,g'\in \CS_{n}(p)$ are such that $\bd^i g = \bd^i g'$ for all $i=0,\dots,n$. The functors $g$ and $g'$ are naturally isomorphic if and only if $\varepsilon(g)=\varepsilon(g')$.
\end{proposition}

\begin{proof}
If $g$ and $g'$ are naturally isomorphic, there is an elementary map $\eta: g(n)\to g'(n)$ such that $\eta(\alpha_i(g)) = \alpha_i(g')$ for all $i=0,\dots, n$. Since the map $\eta$ is elementary, for all $\gamma \in G$
\begin{multline*}
\models q(\alpha_0(g),\dots, \alpha_{n-1}(g),\gamma.\alpha_n(g))
\\
\textrm{ if and only if } \models q(\alpha_0(g'),\dots, \alpha_{n-1}(g'),\gamma.\alpha_n(g')).
\end{multline*}
It follows that $\varepsilon(g)=\varepsilon(g')$.

To establish the other direction, suppose that for some $\gamma\in G$ we have
$$
\models q(\alpha_0(g),\dots, \alpha_{n-1}(g),\gamma.\alpha_n(g)) \textrm{ and } \models q(\alpha_0(g'),\dots, \alpha_{n-1}(g'),\gamma.\alpha_n(g')).
$$
For every proper subset $s\subsetneq n$, the sets $g(s)$ and $g'(s)$ are the same, so we can choose the component maps $\eta_s$, $s\in \Pm(n)$, to be the identity embeddings. It remains to construct the elementary embedding $\eta_n: g(n) \to g'(n)$ that commutes with the transition maps of $g$ and $g'$. The latter requirement means that we are already given the elementary maps $h_i:\Bd ^i(g(n)) \to \Bd^i(g'(n))$ for $i=0,\dots,n$. Namely,
$$
h_i = (g')^{n\setminus \{i\}}_n \circ [g^{n\setminus \{i\}}_n]^{-1}.
$$
The key point is that the union $\bigcup_{i\le n} h_i$ is an elementary embedding if $\varepsilon(g)=\varepsilon(g')$. To see this, first note that $h_i(\alpha_i(g))=\alpha_i(g')$ for each $i\le n$, and in particular $h_n(\gamma.\alpha_n(g))= \gamma. \alpha_n(g')$ since the action is $*$-type definable. Next, since the type $q$ is complete, we have
$$
\alpha_0(g),\dots ,\alpha_{n-1}(g), \gamma.\alpha_n(g)) \equiv  \alpha_0(g'),\dots, \alpha_{n-1}(g'),\gamma.\alpha_n(g')
$$
and so the restriction of $\bigcup_{i\le n} h_i$ to $\alpha_0(g),\dots ,\alpha_{n-1}(g), \gamma.\alpha_n(g)$ is an elementary embedding.
And finally, since $\bd [g] = \dcl(\alpha_0(g),\dots , \gamma.\alpha_n(g))$ and $\bd [g'] = \dcl(\alpha_0(g'),\dots , \gamma.\alpha_n(g'))$, it follows that $\bigcup_{i\le n} h_i: \bd[g]\to \bd [g']$ is elementary. So we can take $\eta_n$ to be any elementary map extending $\bigcup_{i\le n} h_i$.
\end{proof}

\begin{lemma}
\label{lemma b}
The homomorphism $\ov{\varepsilon}: H_n(p)\to G$ is injective.
\end{lemma}

\begin{proof}
By Fact~\ref{Hn_pockets}, an element of $H_n(p)$ has the form $[g-g']$, where $\bd g = \bd g'$. Thus, it is enough to show that if $\bd g=\bd g'$ and $\varepsilon(g-g') =0$, then $g-g'$ is a boundary.
To show the latter, by Fact~\ref{Hn_pockets}, it suffices to construct a natural isomorphism between the functors $g$ and $g'$. But this is exactly what Proposition~\ref{epsilon_prop} does.
\end{proof}

\begin{lemma}
\label{lemma c}
For every $\gamma\in G$, there is $d\in \CZ_n(p)$ such that $\varepsilon(d)=\gamma$.
\end{lemma}

\begin{proof}
Fix $c_0,\dots,c_{n}\models p^{(n+1)}$ and let $g$ be the $n$-simplex with constant transition maps such that $g(s)=\ov{c_s}$ for all $s\subset \{1,\dots,n+1\}$. Let $\mu := \varepsilon(g)$. By the definition of $\varepsilon$, we have $\models q(\alpha_0(g),\dots,\alpha_{n-1}(g), \mu. \alpha_{n}(g))$.

We now define another $n$-simplex $g'$ that has the same boundary as the simplex $g$ (which means that $g'-g$ is a member of $\CZ_n(p)$) and such that $\varepsilon(g'-g)=\gamma$.
Let $w := (-\gamma). \alpha_{n}(g)$. 

By Lemma~\ref{technical lemma}(\ref{action}), there is $\sigma\in \Aut(\CM/\Bd(c_1\dots,c_n))$ taking $\alpha_{n}(g)$ to $(-\gamma). \alpha_{n}(g)$. 
For each $s\subset n$, we let $g'(s):=g(s)$; and all but one transition maps $(g')^s_t$ are the identity embeddings. The one exception is the map $(g')^{n-1}_{n}$: we define
$(g')^{n-1}_{n}:= \sigma\restriction g'(n-1)$. Since $\sigma$ fixes $\Bd (c_0\dots c_{n-1})$, it is easy to check that this transition map is compatible with the remaining transition maps.

We have $\alpha_i(g)=\alpha_i(g')$ for $i=0,\dots,n-1$ and $\alpha_n(g') = (- \gamma).g_{n+1}$. Then  $\models q(\alpha_0(g'),\dots,\alpha_{n-1}(g'), (\mu+\gamma). \alpha_n(g'))$, so $\varepsilon(g')=\mu+\gamma$. Finally,
$\varepsilon (g' - g)=\gamma$.
\end{proof}

As we mentioned above, the main result of the paper now follows easily from the three lemmas.

\begin{proof}[Proof of Theorem~\ref{main}]
By Lemma~\ref{lemma a}, the function $\ov{\varepsilon}: [c]\in H_n(p)\mapsto \varepsilon(c)\in G$ is a well-defined homomorphism. By Lemmas~\ref{lemma b} and~\ref{lemma c}, the homomorphism $\ov{\varepsilon}$ is an isomorphism.
\end{proof}

\section{Proof of Lemma~\ref{technical lemma}}
\label{s:technical lemma}

Throughout this section, we will make the following assumptions: $T = T^{eq}$ is stable, $B = \acl(B)$ is a small subset of the monster model $\CM$ (and without loss of generality, $B = \emptyset$), and the type $p \in S(B)$ has $(\leq n)$-uniqueness but fails $(n+1)$-uniqueness. In addition, we fix at the outset a tuple $c_1, \ldots, c_{n+1}$ realizing $p^{(n+1)}$. Ultimately the choice of the $c_i$'s will not matter, but this will be convenient for some of our definitions.

\subsection{Existence of symmetric systems}
\label{s: sym sys}

For the proof of Lemma~\ref{technical lemma}, it will be useful to have witnesses to the failure of $(n+1)$-uniqueness in $p$ which are invariant under a certain family of automorphisms which we will now describe.

The next definition comes from \cite{GKK3}.

\begin{definition}
\label{n_system}
An \emph{$(n-1)$-symmetric system for $\la c_1, \ldots, c_{n+1}\ra$} is a collection of (generally infinite) tuples $\{\ov{c}_s : s \in [n+1]^{(< n)} \}$ such that:

\begin{enumerate}
\item $\ov{c}_s$ is a tuple enumerating $\acl(B\{c_i : i \in s \})$;
\item for any permutation $\sigma$ of $[n+1]$, there is an automorphism $[\sigma]\in \aut(\CM/B)$ such that $$[\sigma](\ov{c}_s) = \ov{c}_{\sigma(s)}$$ for every $s \in [n+1]^{(< n)}$; and
\item  for any $s \in [n+1]^{(< n)}$ and any two permutations $\sigma$ and $\tau$ of $[n+1]$, $$[\sigma]\circ [\tau](\ov{c}_s) = [\sigma \circ \tau](\ov{c}_s).$$
\end{enumerate}
\end{definition}

Under the assumptions of this section, it is always possible to find an enumeration of the algebraic closures that forms a symmetric system. The following lemma is a strengthening of Lemma~1.20 of \cite{GKK3}.

\begin{lemma}
\label{system_existence}
There is an $(n-1)$-symmetric system for $\la c_1, \ldots, c_{n+1}\ra$.
\end{lemma}

\begin{proof}
This is similar to Lemma~1.20 of \cite{GKK3}, which proved the existence of $(n-1)$-symmetric systems assuming that the theory $T$ has $(\leq n)$-uniqueness, but we need a ``local'' version which uses only the hypothesis that $p$ has $(\leq n)$-uniqueness. However, the proof of Lemma~1.20 of \cite{GKK3} applies verbatim once we have Lemma~1.18 of the cited paper. Lemma~1.18 in turn follows quickly from the property of relative $(n, n+1)$-uniqueness localized to $p$ established in Lemma~\ref{local_rel_nm}. The proof of the latter lemma is in Section 4 of this paper.
\end{proof}

{\bf We fix maps $\{[\sigma]: \sigma \in S_{n+1} \}$ witnessing an  $(n-1)$-symmetric system for $\la c_1,\dots,c_{n+1}\ra$,} as in Lemma~\ref{system_existence}.

\subsection{$n$-ary quasigroupoids and polygroupoids}

The proof of statements (4), (5), and (6) in Lemma~\ref{technical lemma} makes heavy use of some algebraic structures first defined in \cite{GKK3}: $n$-ary polygroupoids. These are $n$-sorted structures with sorts $P_1,\dots , P_n$; for the application here, the sort $P_1$ will be 
the set of realizations of a type $p'$ of a finite tuple inside the algebraic closure of $c\models p$. The elements of sorts $P_i$ will be contained in the algebraic closure of sets of $i$ independent realizations of the type $p'$. There are projection maps between the sorts, and an $(n+1)$-ary relation on the elements of the ``top'' sort $P_n$ -- these will be ultimately used to form the type $q$ in Lemma~\ref{technical lemma}. The property of associativity of an $n$-ary polygroupoid (the difference between a quasi- and a poly-groupoid) is essentially the property (6) in Lemma~\ref{technical lemma}. To make the presentation self-contained, we include the definitions of these structures (quasigroupoid, polygroupoid) and their properties (associativity, local finiteness) in Section 4. 


\begin{Definition}\label{gene.gpoid}
Let $p'(x)\in S(B)$ be a type of a finite tuple contained in the algebraic closure of a realization of $p$. We call a locally finite quasigroupoid 
$(P_1,P_2,\dots,P_n;Q;\pi^2, \ldots, \pi^n)$ an {\em  $n$-ary generic groupoid in $p'$} if
all its sorts are $B$-type-definable and such that
\be
\item
 $Q$ and $\pi^i$ ($i=2,\dots,n$) are $B$-definable;
\item
$P_1$ is the solution set of $p'(x)$;
\item 
for $b_1, ..., b_i\models p'$, $P_i(b_1,...,b_i)\ne \emptyset$ if and only if $\{b_1,...,b_i\}$ is $B$-independent; $P_i(b_1,...,b_i)\subset \ov{b_1,...b_i}$.

\item 
for any any $i \in \{2, \ldots, n\}$, any $B$-independent $b_1,...,b_i$, and any 
$w  \in P_i(b_1,...,b_i)$ the type $ \tp(w / \pi^i(w))$ isolates $\tp(w / \bd(b_1,...,b_i))$,  and if $w'\equiv_{\pi^i(w)} w$ then $w'\in P_i$;

\item   
for any any $i \in \{2, \ldots, n\}$, any $B$-independent $b_1,...,b_i$, and any 
$w,w'  \in P_i(b_1,...,b_i)$ we have $\pi^i(w)=\pi^i(w')$ if and only if $w\equiv_{\ov{b_1,...b_i}} w'$; and

\item
if $Q(v_1,\dots,v_{n+1})$ holds with $v_j\in P_n(c_1,\dots \widehat{c_j},\dots,c_{n+1})$ then $\la c_1,\dots,c_{n+1}\ra $ is Morley in $q$, and 

\item 
associativity of $Q$ holds on any Morley sequence of $(n+2)$-tuples in $q$ (see Definition~\ref{associativity} below).

\ee 
\end{Definition}

Note that in the definition above, since each $\pi^i$ is $B$-definable, for $w\in P_i(\pi(f))$ and any $B$-automorphism $\sigma$ permuting $I$, it follows $\sigma(\pi^i(w))=\pi^i(\sigma(w))$ and $\sigma(w)\in P_i(\sigma(\pi(w)))$. Moreover since $p$ has $(\leq n)$-uniqueness, then for any $u,v\in P_i$ ($i=2,\dots,n$), $u\pi^{i}(u)\equiv_B v\pi^{i}(v)$.

Notice that if we restrict a generic groupoid to the fibers over some infinite Morley sequence $I_0$, then the result will be a connected relatively $I$-definable polygroupoid that were studied in~\cite{GKK3}. Conversely, any relatively $I$-definable polygroupoid in~\cite{GKK3} can be uniquely expanded to a generic $n$-ary groupoid. This allows us to use the facts from~\cite{GKK3}; they are collected in the following two propositions.

The first proposition establishes the existence of generic $n$-ary groupoid that are invariant, in the sense of the two definitions below, under the $(n-1)$-symmetric system of automorphisms constructed in Section~\ref{s: sym sys}.

\begin{Definition}
We say a Morley sequence $\la a_1,\dots,a_{n+1}\ra$ with $a_i\in \ov{c_i}$,
 is {\em $S_{n+1}$-compatible} if 
 $$[\sigma](a_1,\dots, a_{n+1}) =(a_{\sigma(1)},\dots, a_{\sigma(n+1)})$$   for each $\sigma \in S_{n+1}$.
 
 We say a Morley sequence $\la a_1,\dots,a_{i}\ra$ with $a_i\in \ov{c_i}$ $(1\leq i\leq n)$,
 is {\em $S_{n+1}$-compatible} if there are $a_{i+1}...a_{n+1}$ so that $\la a_1,\dots,a_{n+1}\ra$ is 
 $S_{n+1}$-compatible. Obviously if such a compatible $\la a_1,\dots,a_{i}\ra$ is given, then $a_{i+1}...a_{n+1}$ is uniquely determined. 
\end{Definition}

 Let $\sigma_j$ ($j\in [n]$) be the permutation of $[n+1]$ sending  $(1,\dots,n;n+1)$ to $(1,\dots,\hat j,\dots n+1;j)$. Note that the permutations $\sigma_1, \ldots, \sigma_n$ generate the symmetric group on the set $[n+1]$. 

\begin{Definition}
Given an $S_{n+1}$-compatible Morley sequence $\la a_1,\dots,a_{n+1}\ra$, suppose that $(P_1,P_2,\dots,P_n,Q, \pi^2, \ldots, \pi^n)$ is a
 generic $n$-ary groupoid  in $\tp(a_1)$. Then we say  $w\in P_n(a_1,\dots,a_n)$ is {\em compatible with the $[\sigma_j]$'s} if $([\sigma_1](w),\dots [\sigma_n](w),w)$ is a compatible $(n+1)$-tuple from $P_n$.
 \end{Definition}


The following Proposition is adapted from a similar result from \cite{GKK3}.

\begin{proposition}\label{assoc.}
Let $c_1,\dots,c_n\models p^{(n)}$, where $p$ has $(\le n)$-uniqueness, but not $(n+1)$-uniqueness.
\begin{enumerate}
\item The set $\widetilde{c_{1}...c_n}$ is not contained in $\Bd(c_{1}\ldots c_n)$.
\item For every finite tuple $u\in  \widetilde{c_{1}...c_n}$ not contained in $\Bd(c_{1}\ldots c_n)$, there is a type $p_u$, an $S_{n+1}$-compatible Morley sequence $\la a_1,\dots a_{n}\ra $ of finite tuples each realizing $p_u$ with $a_i\in \ov{c_i}$, and a generic $n$-ary groupoid $\CH_u=(P_{u,1},P_{u,2},\dots,P_{u,n}; Q_u; \pi_u^2, \ldots, \pi_u^n)$ in $p_u$ such that $u\in\dcl(w)$  for some $w\in P_{u,n}(a_1,\ldots,a_n)$ compatible with the $[\sigma_j]$'s.
\end{enumerate}
\end{proposition}

\begin{proof}
For part (1), now that we have established the existence of $(n-1)$-symmetric systems, the same construction as in the proof of Proposition~3.13 of \cite{GKK3} can be used to obtain a symmetric failure to $(n+1)$-uniqueness, and for part (2), it is simple to modify the construction of the symmetric failure so that for some $w \in P_{u,n}(a_1, \ldots, a_n),$ $u \in \dcl(w)$. Then the rest of the proof of Theorem~3.3 of \cite{GKK3} can be carried out as before.
\end{proof}

Ultimately, we will construct a directed system of such generic $n$-ary groupoids, but before we do that, we describe a group that can be naturally attached to each such object. The next fact is a compilation of observations from \cite{GKK3}. 
The common setting for Proposition~\ref{mainfact} and Lemma~\ref{typedefassoc} is as follows. The type $p'$ is the type of a finite tuple interalgebraic with a tuple realizing the type $p$. Since the algebraic closures of the realizations of $p$ and $p'$ are the same, the two types have the same uniqueness properties.

\begin{Proposition}\label{mainfact}
 Assume that $p'\in S(B)$ has $(\leq n)$-uniqueness and $B = \acl(B)$.
Let $\mathcal{H} = (I,P_2,\dots,P_n;Q; \pi^2, \ldots, \pi^n)$ be a generic $n$-ary groupoid in $p'$. Then there is a $B$-definable finite abelian
group $G_{\mathcal{H}}$ satisfying the following:

(1)\
For any $B$-independent $b_1,\dots,b_n\models p'$ and $w\in P_n(b_1,\dots,b_n)$,
$G_{\mathcal{H}}$ acts regularly (transitively and faithfully) on the solution set of $\Pi(x;b_1...b_n)$, where $$\Pi(x;b_1...b_n):=\tp(w/\bd(b_1,\dots,b_n)),$$ and so
$|\Pi(\CM;b_1...b_n)|=|G_{\mathcal{H}}|$. 

For any $w,w'\models \Pi(x;b_1...b_n)$ and $\gamma \in G_{\mathcal{H}}$, we have $\dcl(wB)=\dcl(w'B)$ and $\tp(w,\gamma.w/B)=\tp(w',\gamma.w'/B)$. 

Define the relation $\sim$ on the pairs in $(\Pi(\CM;b_1...b_n))^2$ by letting $(w_0,w'_0)\sim(w_1,w'_1)$ if and only if there is $\gamma\in G_{\mathcal{H}}$ such that $w'_j=\gamma.w_j$ ($j=0,1$). Then  $\sim$ is an equivalence relation on  $(\Pi(\CM;b_1...b_n))^2$ and the map $[\ ]: (\Pi(\CM;b_1...b_n)^2/\sim )\to G_{\mathcal{H}}$ sending $[(w_j,w'_j)]$ to $\gamma$ is the unique bijection such that for 
$w, w', w''\models \Pi(x;b_1...b_n)$,
$$[(w, w')]+[(w',w'')]=[(w,w'')].$$ 

The map $[\ ]$ induces the canonical isomorphism between $G_{\mathcal{H}}$ and 
$\aut(\Pi(\CM;b_1...b_n) / \bd (\ov{b}) )$,  sending $\gamma\in G_{\mathcal{H}}$
to the unique $\varphi_{\gamma}$ such that $\gamma.w= \varphi_{\gamma}(w)$ for some (any) $w\models \Pi(x;b_1...b_n)$. 
In other words $G_{\mathcal{H}}$ uniformly and canonically binds all the automorphism groups $\aut(\Pi(\CM;b'_1...b'_n)$ with $b'_1...b'_n \models (p')^{(n)}$. 

(2)\ 
Assume  
$Q(w_1,...,w_{n+1})$ holds. Then for any $(\gamma_1, \ldots, \gamma_{n+1}) \in G_{\mathcal{H}}^{n+1}$, we have 
 $$Q(\gamma_1.w_1, \gamma_2.w_2, \ldots, \gamma_{n+1} .w_{n+1}) \Leftrightarrow \sum_{i=1}^{n+1} (-1)^i \gamma_i = 0.$$

\end{Proposition}

\begin{proof}
The construction of this group $G_{\mathcal{H}}$ and the proof of all of the properties listed above is given by the series of lemmas from Lemma~3.15 to Corollary~3.27 in \cite{GKK3} (there it was called simply ``$G$''). Some steps in this proof use the property of relative $(k,n)$-uniqueness for values of $k \leq n$, but for these steps we can apply Lemma~\ref{local_rel_nm}, since throughout we assume only $(\leq n)$-uniqueness for the type $p'$ and not for $T$.
\end{proof}

We state a new observation which will be crucial for the proof of Lemma \ref{typedefcommut} below.

\begin{Lemma}\label{typedefassoc}
Assume that $p'\in S(B)$ has $(\leq n)$-uniqueness, $B = \acl(B)$, and  that  a generic $n$-ary groupoid $(I,P_2,\dots,P_n;Q;\pi^2, \ldots, \pi^n)$  in $p'$ is given.
Let $\Phi_{p'}(x_1,\dots,x_{n+1})$ be the partial type over $B$ saying that for some  $x^k_{j}$ ($k=1,\dots, n+2$, 
 $j=1,\dots,n+1$) with $x_j=x^{n+2}_j$, and some 
  Morley sequence $\la a_1,\dots,a_{n+2}\ra$
 in $q$, 
 \be\item 
 $(x_1,\dots, x_{n+1})$ is compatible from $P_n$, and $x_1...x_{n+1}\equiv_B x^k_1...x^k_{n+1}$; 
 \item$x^k_j=x^{j+1}_k$ for all $1\leq k\leq j\leq n+1$; and
 \item  $\pi(x^k_j)=(d_1,\dots,\widehat{d_j},\dots,d_{n+1})$
 where $d_1\dots d_{n+1}=a_1\dots \widehat{a_k} \dots a_{n+2}.$
\ee
Then 
 $\models \Phi_{p'}(v_1,\dots, v_{n+1})$ if and only if
 there is  some $\CL(B)$-formula $Q'(x_1,\dots,x_{n+1})$ realized by compatible $(v_1,\dots, v_{n+1})$
 from $P_n$,
such that $(I,\ldots , P_n;Q';\pi^i )$ 
 forms a generic  $n$-ary groupoid in $p'$ (where we have just replaced $Q$ by $Q'$).

\end{Lemma}
\begin{proof}
$(\Leftarrow)$: This direction is clear due to the associativity of $Q'$ and the $(\leq n)$-uniqueness property for $p'$.

$(\Rightarrow)$:
Assume $\Phi_{p'}(v_1,\dots, v_{n+1})$ holds. So there  is some Morley sequence 
$\la a_1,\dots a_{n+1}\ra$ in $p'$ such that 
$\pi(v_j)=(a_1,\dots,\widehat{a_j},\dots,a_{n+1})$ is witnessed by a formula, say
 $\theta (x_j;y_1,\dots,\widehat{y_j},\dots,y_{n+1})\in \CL(B)$. Moreover notice  that 
 $v_j\in\dcl(v_1\dots\widehat{ v_j}\dots v_{n+1};B)$ for each $j\in [n+1]$, since due to $(\leq n)$-uniqueness there is $v'_j$ such that 
 $v_1\dots v'_j\dots v_{n+1}\equiv_B w_1\dots w_{n+1}$ where $(w_1,\dots, w_{n+1})\models Q$, so $v'_j\in \dcl(v_1\dots\widehat{ v_j}\dots v_{n+1};B)$, and 
 by Fact \ref{mainfact}(1), $v_j$ and $v'_j$ are interdefinable over $B$. 
 Hence the fact that 
 $v_j\in\dcl(v_1,\dots\widehat{ v_j}\dots v_{n+1};B)$ is witnessed  by some $\CL(B)$-formula 
$$\alpha_j(x_j;x_1,\dots,\widehat{x_j},\dots, x_{n+1}).$$ We also choose a formula 
$\beta(x_1,\dots, x_{n+1})\in \CL(B)$ which says ``the $(n+1)$-tuple $(x_1, \ldots, x_{n+1})$ is compatible.'' 
Now we let $Q'(x_1,\dots,x_{n+1})$ be the formula

\begin{multline*}
\beta(x_1,\dots, x_{n+1}) \wedge
\\
\exists y_1\dots y_{n+1}\bigwedge^{n+1}_{j=1} (\theta (x_j;y_1,\dots,\widehat{y_j},\dots,y_{n+1})\wedge \alpha_j(x_j;x_1,\dots,\widehat{x_j},\dots, x_{n+1})),
\end{multline*} 



 
\noindent realized by
$(v_1,\dots,v_{n+1})$. (In fact, $Q'(v_1,\dots,v_{n+1})$ together with a type over $B$ saying 
that for some Morley $\la a_1,\dots a_{n+1}\ra$ in $p'$, $\pi(v_j)=(a_1,\dots,\widehat{a_j},\dots,a_{n+1})$ for each $j\in[n+1]$,
determines $\tp(v_1,\dots, v_{n+1}/B)$.)
Now it suffices to check that $Q'$ satisfies the associativity on a Morley sequence of 
$(n+2)$-tuple in $p'$. 

Since $(v_1,\dots, v_{n+1})$ satisfies (1),(2), and (3), there exist elements 
$v_j^k$ and a $B$-Morley sequence  $\la a_1,\dots a_{n+2}\ra$ satisfying these conditions (where $v_j^k$ plays the role of $x_j^k$). 
In particular, $\models Q'(v^k_1,\dots,v^k_{n+1})$ for all $k\in[n+2]$. 
On the other hand, assume that 
there exist $u^k_{j}$ ($k=1,\dots, n+2$, 
 $j=1,\dots,n+1$) such that for each $k\in [n+2]$,
 $(u^k_1,\dots, u^k_{n+1})$ is compatible from $P_n$;  
 $u^k_j=u^{j+1}_k$ for all $1\leq k\leq j\leq n+1$; and $\pi(x^k_j)=(d_1,\dots,\widehat{d_j},\dots,d_{n+1})$
 where $d_1\dots d_{n+1}=a_1\dots \widehat{a_k} \dots a_{n+2}.$ Moreover assume 
 $Q'( u^k_1,\dots, u^k_{n+1})$ holds for each $k\in [n+1]$. We want to verify that
 it holds as well when $k=n+2$ (the other cases being similar). 
 
 Note first that due to $(\leq n)$-uniqueness for $p'$, it follows that $v_1\dots v_{n}\equiv_B u^k_1\dots u^k_{n}$
 for all $k\in [n+2]$. Since $Q'$ isolates $\tp(v_{n+1}/v_1\dots v_{n}B)$, this implies
 $$v_1\dots v_{n+1}\equiv_B u^k_1\dots u^k_{n+1}$$
 for all $k\in [n+1]$.  Furthermore, again by $(\leq n)$-uniqueness, we have 
 $$\la v^k_1\dots v^k_{n+1}|\ k=1,\dots,n\ra \equiv_B \la u^k_1\dots u^k_{n+1}|\ k=1,\dots,n\ra.$$
   Then due to the $Q'$-relation on $(v^{n+1}_1,\dots, v^{n+1}_{n+1})=(v_{n}^1,\dots, v_{n}^{n+1})$, and 
   on $(u^{n+1}_1,\dots, u^{n+1}_{n+1})=(u_{n}^1,\dots, u_{n}^{n+1})$, it clearly follows that 
    $$\la v^k_1\dots v^k_{n+1}|\ k=1,\dots,n+1\ra \equiv_B \la u^k_1\dots u^k_{n+1}|\ k=1,\dots,n+1\ra.$$
    Then since $Q'(v^{n+2}_1,\dots,v^{n+2}_{n+1})$ holds, so does 
    $Q'(u^{n+2}_1,\dots,u^{n+2}_{n+1})$.
\end{proof}

\subsection{Constructing a directed system of generic $n$-ary groupoids.}

The first key observation, stated in Proposition~\ref{assoc.}, is that since $p$ fails $(n+1)$-uniqueness, there is some 
$$
w\in \widetilde{c_{1} \ldots c_n}\setminus \Bd(c_{1}\ldots c_n).
$$  
If we could pick a single $w \in \CM$  which was interdefinable with all of $\widetilde{c_{1} \ldots c_n}$, then letting $q = \tp( w / \Bd(c_{1}\ldots c_n))$, we would have $$ G \cong G_q := \Aut(q(\CM) / \Bd(c_1 \ldots c_n)),$$ and $G$ would be finite since $q$ is algebraic.

 However, in general there is no such $w$ which is maximal with respect to definable closure, and the group $G$ will be constructed as an inverse limit of finite groups $G$ as above. It will take some work to construct an appropriate index set for the inverse limit and check that this has all the properties we need.


The next Lemma is a generalization and adaptation of Claim~2.9 from \cite{GKK2}, and it can be proved in the same manner (using Proposition~\ref{assoc.} above).

\begin{Lemma}
\label{index_set}
There is a directed partially ordered set $(J, \leq)$, with $|J| \leq |T|$, and families $c_{1,u}, \ldots, c_{n+1,u}$, $p_u$, $\mathcal{H}_u$, and $w_u$ indexed by $u \in J$, satisfying all of the following properties for every $u \in J$:

\begin{enumerate}
\item $c_{1,u}, \ldots, c_{n+1,u}$ is a Morley sequence over $B$ such that each $c_{i,u}$ is interalgebraic with $c_i$ over $B$;
\item $p_u = \tp(c_{1,u} / B)$;
\item $\mathcal{H}_u = (I_u, P_{2,u}, \ldots, P_{n,u}; Q_u; \pi^2_u, \ldots, \pi^n_u)$ is an $n$-ary generic groupoid in $p_u$;
\item $w_u \in P_{n,u}(c_{1,u}, \ldots, c_{n,u})$;
\item $w_u$ is compatible with the $[\sigma_j]$'s;
\item If $u \leq v$, then $c_{1,u} \in \dcl c_{1,v}$, $w_u \in \dcl(w_v)$, and $$c_{1,u} c_{1,v} \equiv c_{2,u} c_{2,v} \equiv \ldots \equiv c_{n+1,u} c_{n+1,v};$$ and
\item $$\widetilde{c_1 \ldots c_n} = \dcl_B \left( \bigcup_{u\in J} w_u \right).$$
\end{enumerate}

\end{Lemma}

The following notation is designed to match the both the notation in Proposition~\ref{mainfact} and the notation for the selector function $\Pi$ in Lemma~\ref{technical lemma}. It is also convenient to have symbols for the sets of realizations of certain types. These elements explicitly witness the failure of $(n+1)$-uniqueness for the type $p$ and have a lot of built-in symmetry with respect to the automorphisms $[\sigma_j]$; so we chose the notation ``$SW$'' for ``symmetric witness''.

\begin{notation}
\label{SW}
If $u \in J$, $$\Pi_u(x_u; y_u) = \tp(w_u; \pi^n_u(w_u))$$ 


and if $d \equiv_B \pi^n_u(w_u)$, then $$SW_u(d) = \Pi_u(\mathcal{M}; d);$$ and $$SW_u = \bigcup_{d \equiv_B \pi^n_u(w_u)} SW_u(d).$$
\end{notation}

Then, as in Lemma~2.10 of \cite{GKK2}, we obtain:

\begin{Lemma}
For any $u, v \in J$ such that $u \leq v$, there are surjective functions $$\tau_{v,u} : p_v(\mathcal{M}) \rightarrow p_u(\mathcal{M})$$ and $$\rho_{v,u}: SW_v \rightarrow SW_u $$ which are relatively $B$-definable (that is, their graphs are the intersections of $B$-definable relations with the product of the domain and the range) and satisfy all of the following:
\begin{enumerate}
\item $\tau_{v,u}(c_{i,v}) = c_{i,v}$,
\item $\rho_{v,u}(w_v) = w_u$,
\end{enumerate}
and whenever $u \leq v \leq w$,
\begin{enumerate}
\setcounter{enumi}{2}
\item $\tau_{v,u} \circ \tau_{w,v} = \tau_{w,u}$, and
\item $\rho_{v,u} \circ \rho_{w,v} = \rho_{w,u}$.
\end{enumerate}
\end{Lemma}

The next Lemma says that we can make slight changes in the definable relations $Q_u$ in order to make them coherent with the projection maps $\rho_{v,u}$. This will be crucial below to ensure that we can take an inverse limit of the groups $G_u$.

\begin{Lemma}
\label{typedefcommut}
There is a family $\langle Q'_u : u \in J \rangle$ of formulas over $B$ such that, after replacing $\mathcal{H}_u$ by $\mathcal{H}'_u := (I_u, \ldots; Q'_u; \pi^2_u, \ldots, \pi^n_u)$ (only changing $Q_u$ to $Q'_u$), we have:

\begin{enumerate}
\item The  family $\langle \mathcal{H}'_u : u \in J \rangle$ is still an $n$-ary generic groupoid; and
\item Whenever $u \leq v$, $(d_1, \ldots, d_{n+1}) \models (p_v)^{(n+1)}$, and $w_1, \ldots, w_{n+1} \in P_v$ are elements satisfying $$\pi_v(w_i) = (d_1, \ldots, \widehat{d_i}, \ldots, d_{n+1})$$ and $$Q'_v(w_1, \ldots, w_{n+1}),$$ then $$Q'_u(\rho_{v,u}(w_1), \ldots, \rho_{v,u}(w_{n+1})).$$
\end{enumerate}

\end{Lemma}

\begin{proof}

For each $u \in J$, since $$([\sigma_1](w_u), [\sigma_2](w_u), \ldots, [\sigma_n](w_u), w_u)$$ is a compatible tuple from a generic $n$-ary groupoid, by Lemma~\ref{typedefassoc}, there is a relation $Q'_u$ which is satisfied by this tuple such that $\mathcal{H}'_u = (I_u, \ldots; Q'_u; \pi^2_u, \ldots, \pi^n_u)$ is still a generic $n$-ary groupoid.

To check the second condition in the Lemma, suppose that $u, v \in J$ and $u \leq v$ and that $w_1, \ldots, w_{n+1} \in P_v$ satisfy $Q'_v(w_1, \ldots, w_{n+1})$ and $\pi_v(w_i) = (d_1, \ldots, \widehat{d_i}, \ldots, d_{n+1})$ for some $d_1, \ldots, d_{n+1}$.

\begin{claim}
\label{elem_map}
There is an elementary map $\varphi \in \aut(\CM / B)$ such that $\varphi(d_i) = c_{i,v}$ for $i \in \{1, \ldots, n+1\}$ and $\varphi(w_i) = [\sigma_i](w_v)$ for $i \in \{1, \ldots, n\}$.
\end{claim}

\begin{proof}
As in Definition~3.1 of \cite{GKK}, we say that a type $p$ has \emph{relative $n$-uniqueness} just in case $p$ has relative $(n,n)$-uniqueness. Now by the proofs of Theorem~3.10 and Proposition~3.14 of \cite{GKK}, the fact that $p$ has $(\leq n)$-uniqueness implies that $p$ has $B(n)$ and thus relative $n$-uniqueness (in fact, just the $n$-uniqueness and $(n-1)$-uniqueness of $p$ would be sufficient for this step).

Since $\{c_{1,v}, \ldots, c_{n+1,v}\}$ and $\{d_1, \cdots, d_{n+1}\}$ are Morley sequences over $B$ in the same type, there is an elementary map $\varphi_0 \in \aut(\CM / B)$ such that $\varphi_0(d_i) = c_{i,v}$ for each $i \in \{1, \ldots, n+1\}$. By relative $n$-uniqueness of $p$ over $\acl(B d_1)$, there is an elementary map $\varphi_1 \in \aut(\CM / \acl(B c_{1,v}))$ such that $\varphi_1$ fixes $\acl_B(c_{1,v}, \ldots, \widehat{c_{i,v}}, \ldots, c_{n+1,v} \}$ pointwise for every $i \in \{2, \ldots, n+1\}$ and such that $$\varphi_1(\varphi_0(w_i)) = [\sigma_i] (w_v)$$ for every $i \in \{1, \ldots, n\}$. Now $\varphi := \varphi_1 \circ \varphi_0$ is the map we seek.

\end{proof}

Now since $\varphi$ is elementary, it preserves the $Q'_v$ relation, and therefore $\varphi(w_{n+1}) = w_v$. 

Likewise, since the elementary maps $[\sigma_i]$ commute with the definable map $\rho_{v,u}$,  $$(\rho_{v,u}([\sigma_1](w_v)), \ldots, \rho_{v,u}([\sigma_n](w_v), \rho_{v,u}(w_v))$$ $$= ([\sigma_1](\rho_{v,u}(w_v)), \ldots, [\sigma_n](\rho_{v,u}(w_v)), \rho_{v,u}(w_v))$$ $$= ([\sigma_1](w_u), \ldots, [\sigma_n](w_u), w_u),$$ and $Q'_u$ holds of the final tuple by definition. That is, the $(n+1)$-tuple $([\sigma_1](w_v), \ldots, [\sigma_n](w_v), w_v)$ satisfies the formula $$Q'_u(\rho_{v,u}(x_1), \ldots, \rho_{v,u}(x_{n+1}))$$ (assigning $[\sigma_1](w_v)$ to $x_1, \ldots,$ $[\sigma_n](w_v)$ to $x_n$, and $w_v$ to $x_{n+1}$). Taking preimages via the elementary map $\varphi$, we see that $(w_1, \ldots, w_{n+1})$ also satisfies this formula, which is what we wanted to prove.









\end{proof}

\subsection{Proof of Lemma~\ref{technical lemma}}

\begin{proof}
 First, let $\overline{w}$ be a (possibly infinite) tuple listing $\{w_u : u \in J\}$, and let $q(x_0, \ldots, x_n)$ be the complete $*$-type over $B$ of $$([\sigma_1](\overline{w}), [\sigma_2](\overline{w}), \ldots, [\sigma_n](\overline{w}), \overline{w}).$$ 

Let $r$ be the complete $*$-type $$r(x; y) := \tp(\ov{w} ; \bd(c_1, \ldots, c_n) / B).$$ Given any $f \in \mathcal{S}_{n-1}(p)$, note that the infinite tuple $\bd [f]$ can be ordered so that $\bd [f] \equiv \bd(c_1, \ldots, c_n)$: this follows from $(\leq n)$-uniqueness of the type $p$ by the same argument as in the proof of Claim~\ref{elem_map} above. Pick any such ordering of $\bd [f]$, and let $\Pi f$ be $r(x ; \bd [ f])$.

Given $f \in \mathcal{S}_{n-1}(p)$, let $\alpha(f)$ be \emph{any} realization of $\Pi f$.


By Proposition~\ref{mainfact} above, for each $u \in J$, the generic $n$-ary polygroupoid $\mathcal{H}'_u$ has a finite abelian ``binding group'' $G_{\mathcal{H}'_u}$, which we will denote by simply $G_u$. Each $G_u$ acts as described in Proposition~\ref{mainfact} on the set of all realizations of $SW_u(d)$ for any $d \equiv_B \pi^n_u(w_u)$.

Now for $u, v \in J$ such that $u \leq v$, we can define a surjective group homomorphism $\chi_{v,u} : G_v \rightarrow G_u$ as follows: if $\gamma \in G_v$, then for any $w \in SW_v(\pi^n_v(w_v))$, we have that $\gamma . w$ is also an element of $SW_v(\pi^n_v(w_v))$. Now we define $\chi_{v,u}(\gamma)$ to be the unique element of $G_u$ such that $$\chi_{v,u}(\gamma) . \rho_{v,u}(w) = \rho_{v,u}(\gamma . w).$$ This definition of $\chi_{v,u}$ does not depend on the choice of $w$ since for any other $w' \in SW_v(\pi^n_v(w_v))$, we have $(w, \gamma . w) \equiv_B (w', \gamma . w')$, and so the $B$-definability of $\rho_{v,u}$ implies that $$(\rho_{v,u}(w), \rho_{v,u}( \gamma . w)) \equiv_B (\rho_{v,u}(w'), \rho_{v,u}( \gamma . w')),$$ whence $\chi_{v,u}(\gamma) . \rho_{v,u}(w') = \rho_{v,u}(\gamma . w').$

The fact that $\chi_{v,u}$ is a group homomorphism now follows directly from its definition, and it is surjective since $G_v$ acts transitively on $SW_v(\pi^n_v(w_v))$.

Finally, our group $G$ will be the inverse limit of $\langle G_u : u \in J \rangle$. If we think of each $\alpha \in G$ concretely as a tuple $\langle \gamma_u : u \in J \rangle$ such that whenever $u \leq v$, we have $\chi_{v,u}(\gamma_v) = \gamma_u$, then we can define the action of $\gamma$ as follows: if $\overline{w}' \models \Pi f$, then $\overline{w}'$ enumerates some elements $\{w'_u : u \in J\}$ in such a way that $(\overline{w}', \bd [f]) \equiv_B (\overline{w}, \bd(c_1, \ldots, c_n))$, and we define $\alpha \cdot \overline{w}'$ as the enumeration (in the same order) of the elements $\{ \gamma_u \cdot w'_u : u \in J\}$. The remaining desired properties of the action (Lemma~\ref{technical lemma} parts (3)-(6)) now follow from Proposition~\ref{mainfact}. 
\end{proof}

\section{Symmetric systems and polygroupoids under local uniqueness assumptions}

This section verifies the statements that are necessary to strengthen the results of~\cite{GKK3} to the form needed for this paper. 

\subsection{Symmetric system}

\begin{lemma}
\label{local_rel_nm} Suppose that $T$ is stable, $p \in S(B)$ (with $B = \acl^{eq}(B)$) has $n$-uniqueness, and $m \geq n$. Then $p$ has the following property:

For any Morley sequence $a_1, \ldots, a_m$ over $B$ such that each $a_i$ is interalgebraic (over $B$) with a realization of $p$, and for any set of maps $\{\varphi_u : u \subset_{n-1} [m] \}$ such that

\begin{enumerate}
\item $\varphi_u$ is an automorphism of $\acl(\{a_i : i \in u \} \cup B)$, and
\item For all $v \subsetneq u$, $\varphi_u$ acts as the identity on $\acl(\{a_i : i \in v \} \cup B)$,
\end{enumerate}

we have that $\bigcup_u \varphi_u$ is an elementary map.

We call this property \emph{relative $(n,m)$-uniqueness localized to $p$}.

\end{lemma}

\begin{proof}
This follows from the next two Claims, which are localized versions of two lemmas from \cite{GKK}.

\begin{claim}
\label{loc_Bn}
If $p$ has $n$-uniqueness, then $p$ has the property $B(n)$ (see Definition~3.1 of \cite{GKK} -- to say that $p$ has this property means that it is true when the tuples $a_i$ are interalgebraic over $B$ with realizations of $p$).
\end{claim}

\begin{proof}
The same argument as in the proof of Lemma~3.3 of \cite{GKK} applies without changes.
\end{proof}

\begin{claim}
\label{rel_nm_claim}
If $p$ has $B(n)$, then relative $(n,m)$-uniqueness localized to $p^{(k)}$ holds for any $k$.
\end{claim}

\begin{proof}
By the main result of \cite{K}, since $p$ has $n$-uniqueness, so does $p^{(k)}$ for any $k$. So by Claim~\ref{loc_Bn}, every $p^{(k)}$ also has $B(n)$. Now we can carry out the same argument as in the proof of Lemma~4.4 of \cite{GKK}, which goes by induction on $m$. This argument proves $(n,m+1)$-uniqueness of $p^{(k)}$ by applying relative $(n,m)$-uniqueness to systems of elementary maps defined over independent sets of the form $\{b_1, \ldots, b_m\}$ where $b_i$ realizes $p^{(k)}$ or $p^{(2k)}$, which holds by the induction hypothesis.
\end{proof}

\end{proof}

\subsection{Polygroupoids}

\begin{definition}
\label{comp_tuples}
Let $M=(I, P_2,\dots,P_n,\pi^2,\dots,\pi^n)$ be a structure with sorts $P_i$, $i=2,\dots,n$ and
functions $\pi^k:P_k\to (P_{k-1})^k$, $k=2,\dots,n$; sometimes we may use ``$P_1$'' to refer to the sort $I$. We use the symbol $\pi^k_i(w)$ to refer to the $i$th element of the tuple $\pi^k(w)$.
\be\item
We say that a tuple $(w_1,\dots,w_{k+1})\in (P_k)^{k+1}$ is \emph{compatible} if
\begin{enumerate}
\item
$k=1$ and $w_1\ne w_2$, or
\item
$k\ge 2$ and $\pi^k_{i}(w_j)=\pi^k_{j-1}(w_i)$ for all $1\le i<j\le k+1$.
\end{enumerate}

We say that a $k$-tuple $$(w_1,\dots, w_{\ell-1}, \widehat{\cdot}, w_{\ell+1}, \dots,w_{k+1}) \in (P_k)^{k}$$
with the deleted term number $\ell(\leq k+1)$, 
is \emph{partially compatible} if either $k=1$, or  $k\geq 2$ and 
$\pi^k_{i}(w_j)=\pi^k_{j-1}(w_i)$ for all $1\le i(\ne \ell)<j(\ne \ell)\le k+1.$

\item Assume that for every $w \in P_k$, the image $\pi^k(w)$ forms a compatible tuple:
Then for any $i \in \{2, \ldots, n\}$ and any $w \in P_i$, we iteratively say that \emph{$w$ is over $(a_1, \ldots, a_i) \in I^{(i)}$} if,
\begin{enumerate}
\item $i = 2$ and $(a_2, a_1) = \pi^2(w)$, or
\item $i > 2$ and for every $j \in \{1, \ldots, i\}$, $\pi^i_{j}(w)$ is over $(a_1, \ldots, \widehat{a}_j, \ldots, a_i)$;
\end{enumerate}
For any $(a_1, \ldots, a_i) \in I^{(i)}$, we denote by $P_i(a_1, \ldots, a_i)$ the set of all $w \in P_i$ which are over $(a_1, \ldots, a_i)$. If $w \in P_i(a_1, \ldots, a_i)$, we also write ``$\pi(w) = (a_1, \ldots, a_i)$'';
Note that then the sort $P_i$ is the disjoint union of all the ``fibers'' $P_i(a_1, \ldots, a_i)$ where $(a_1, \ldots, a_i) \in I^{(i)}$.

Given $w\in P_i(a_1,\dots,a_i)$, and a $j$-tuple $s$ $(0<j\leq i)$ 
 of increasing numbers from $[i]$, we write $\pi_s(w)$ to denote the unique element in
 $P_{|s|}(a_j|\ j\in s )$ which is an image of $w$ under a composition of $\pi^2,\dots,\pi^i$. For example, $\pi_{\la 1, \ldots, i \ra}(w)=w$,
 $\pi_{\la 1\ra}(w)=a_1$,  and $\pi_{\la 1,\widehat{2},\widehat{3},...,i\ra}(w)=\pi^{i-1}_{2}(\pi^i_{2}(w))$.
\ee
\end{definition}

%
%

\begin{definition}
\label{quasigroupoid}
If $n \geq 2$, an \emph{$n$-ary quasigroupoid} is a structure $\mathcal{H} = (I, P_2, \ldots, P_{n-1}, P, Q)$ with $n$ disjoint sorts $I = P_1, P_2, \ldots, P_n = P$ equipped with an $(n+1)$-ary relation $Q \subseteq P^{n+1}$ and a system of maps $\langle \pi^k : 2 \leq k \leq n \rangle$ satisfying the following axioms:

\begin{enumerate}
\item (Coherence) For each $k \in \{2, \ldots, n\}$, the function $\pi^k$ maps an element $w\in P_k$ to a
compatible $k$-tuple
$(\pi^k_{1}(w),\dots, \pi^k_{k}(w))\in (P_{k-1})^k $.

\item
(Compatibility and $Q$)
If $Q(w_1, \ldots, w_{n+1})$ holds, then $(w_1, \ldots, w_{n+1})$ is a compatible $(n+1)$-tuple
of elements of $P_n$.

\item (Uniqueness of horn-filling) Whenever $Q(w_1, \ldots, w_{n+1})$ holds, then for any $i \in \{1, \ldots, {n+1}\}$, $w_i$ is the unique element $x \in P$ which satisfies $$Q(w_1, \ldots, w_{i-1}, x, w_{i+1}, \ldots, w_{n+1}).$$
\end{enumerate}


We say the $n$-ary quasigroupoid 
$\mathcal{H} = (I, \ldots, P, Q)$ 
is \emph{locally finite} if for every $z \in I^{(n)}$, the set $P(z)$ is finite.
\end{definition}

%
%

\begin{definition}
\label{associativity}
If $\mathcal{H} = (I, P_2, \ldots, P_n, Q)$ is an $n$-ary quasigroupoid, we say that $\mathcal{H}$ is an \emph{$n$-ary polygroupoid} if it satisfies the following condition:

(Associativity) Suppose that 
$\{w^i_j \mid 1 \leq i \leq n+2,\ 1\le j\le n+1\}$ is a 
collection of elements in $P$ such that for each $i=1,\dots,n+2$
the elements $\{w^i_j\mid 1\le j\le n+1\}$ are compatible
and such that $w^i_j=w^{j+1}_i$ for all $1\le i\le j\le n+1$.

For each $\ell=1,\dots,n+2$, if $Q(w^i_1,\dots,w^i_{n+1})$ 
hold for all $i\in \{1,\dots,n+2\}\setminus \{\ell\}$, then 
$Q(w^\ell_1,\dots,w^\ell_{n+1})$ holds too.
\smallskip

If associativity holds for any compatible such tuples $w^i_j$
with 
$\pi(w^i_j)=(d^i_1,\dots,\widehat{d^i_j},\dots,d^i_{n+1})\in I^{(n)}$
 where $d^i_1\dots d^i_{n+1}=c_1\dots \widehat{c_i} \dots c_{n+2},$ then we say
 {\em the associativity of $Q$ holds on $(c_1,\dots,c_{n+2})$}.
\end{definition}

\end{document}